%% file: main.tex
\documentclass[sn-mathphys]{sn-jnl}% Math and Physical Sciences Reference Style
\usepackage{graphicx}%
\usepackage{multirow}%
\usepackage{amsmath,amssymb,amsfonts}%
\usepackage{amsthm}%
\usepackage{mathrsfs}%
\usepackage[title]{appendix}%
\usepackage{xcolor}%
\usepackage{textcomp}%
\usepackage{manyfoot}%
\usepackage{booktabs}%
\usepackage{algorithm}%
\usepackage{algorithmicx}%
\usepackage{algpseudocode}%
\usepackage{listings}%
\usepackage{centernot}  % nice /not commands
\usepackage{xfrac}      % nice fractions
\usepackage{bbm}        % indicators
\usepackage[capitalize]{cleveref}  % smart referencing
\usepackage{parskip}    % paragraph formatting

\input{macros/math}

\newtheorem{theorem}{Theorem}%
\newtheorem{corollary}[theorem]{Corollary}% 
\newtheorem{lemma}[theorem]{Lemma}% 
\newtheorem{definition}[theorem]{Definition}%

\begin{document}

\title[Faster Convergence of Stochastic Accelerated Gradient Descent under Interpolation]%
{Faster Convergence of Stochastic Accelerated Gradient Descent under Interpolation}


\author*[1]{\fnm{Aaron} \sur{Mishkin}}\email{amishkin@cs.stanford.edu}
\author[2]{\fnm{Mert} \sur{Pilanci}}\email{pilanci@stanford.edu}
\author[3,4]{\fnm{Mark} \sur{Schmidt}}\email{schmidtm@cs.ubc.ca}

\affil[1]{\orgdiv{Department of Computer Science}, \orgname{Stanford University}}
\affil[2]{\orgdiv{Department of Electrical Engineering}, \orgname{Stanford University}}
\affil[3]{\orgdiv{Department of Computer Science}, \orgname{University of British Columbia}}
\affil[4]{\orgdiv{Canada CIFAR AI Chair}, \orgname{Amii}}

%\affil[2]{\orgdiv{Department}, \orgname{Organization}, \orgaddress{\street{Street}, \city{City}, \postcode{10587}, \state{State}, \country{Country}}}

%\affil[3]{\orgdiv{Department}, \orgname{Organization}, \orgaddress{\street{Street}, \city{City}, \postcode{610101}, \state{State}, \country{Country}}}

\abstract{
    \textcolor{red}{
        This preprint has a significant bug in its proofs.
        In particular, the claim that the generalized stochastic AGD scheme in
        \cref{eq:stochastic-agd} with the map $m$ set to be stochastic gradient
        update is equivalent to the standard ``momentum'' version of stochastic
        AGD in \cref{eq:stochastic-agd-simple} \textbf{does not hold}.
        Unfortunately, the bug invalidates the main conclusion of the paper ---
        namely that the dependence on the strong growth constant can be
        improved from \( \rho \) to \( \sqrt{\rho} \) for stochastic Nesterov
        acceleration.
        We are currently do not know if this issue can be corrected to obtain
        the desired \( \sqrt{\rho} \) dependence or if \( \rho \) is in fact
        tight.
        We will continue to explore this issue and post an updated version of
        the preprint if we obtain any new results.
        Note that all results hold as stated for the semi-stochastic scheme in
        \cref{eq:stochastic-agd}.
        Please see \cref{app:bug} for further details on the bug.
    }

    We prove new convergence rates for a generalized version of stochastic Nesterov acceleration
    under interpolation conditions.
    Unlike previous analyses, our approach accelerates any stochastic gradient
    method which makes sufficient progress in expectation.
    The proof, which proceeds using the estimating sequences framework, applies to
    both convex and strongly convex functions and is easily specialized to
    accelerated SGD under the strong growth condition.
    In this special case, our analysis reduces the dependence on the strong growth
    constant from \( \rho \) to \( \sqrt{\rho} \) as compared to prior work.
    This improvement is comparable to a square-root of the condition number in the
    worst case and address criticism that guarantees for stochastic acceleration
    could be worse than those for SGD.
}

\maketitle

\section{Introduction}\label{sec:intro}

\input{sections/intro}

\section{Assumptions}\label{sec:assumptions}

\input{sections/assumptions}

\section{Convergence of
  Stochastic AGD}\label{sec:convergence}

\input{sections/convergence}

\section{Comparison to
  Existing Rates}\label{sec:discussion}

\input{sections/discussion}

\section{Conclusion}

\input{sections/conclusion}

\backmatter

\bmhead{Acknowledgments}

We would like to thank Frederik Kunstner, Victor Sanchez-Portella, and Sharan
Vaswani for many insightful discussions.
\textcolor{red}{We thank Chia-Yu Hsu for their help in discovering the bug in
this preprint.}
%%\pacs[JEL Classification]{D8, H51}%%\pacs[MSC Classification]{35A01, 65L10, 65L12, 65L20, 65L70}%\bmhead{Supplementary information}%TODO: acknowledge any supplementary files here.

\bmhead{Funding}

Aaron Mishkin was supported by NSF GRF Grant No.
DGE-1656518 and by NSERC PGS D Grant No.
PGSD3-547242-2020.
Mert Pilanci was supported by the NSF under Grant ECCS-2037304 and Grant
DMS-2134248, by an NSF CAREER Award under Grant CCF-2236829, by the U.S.
Army Research Office Early Career Award under Grant W911NF-21-1-0242, and by
the Stanford Precourt Institute.
Mark Schmidt was partially supported by the Canada CIFAR AI Chair Program and
NSERC Discovery Grant No.
RGPIN-2022-03669.

%\begin{itemize}
%\item Conflict of interest/Competing interests (check journal-specific guidelines for which heading to use)
%\item Ethics approval 
%\item Consent for publication
%\item Availability of data and materials
%\item Authors' contributions

%\noindent
%If any of the sections are not relevant to your manuscript, please include the heading and write `Not applicable' for that section. 

\newpage

\bibliography{refs}% common bib file

\newpage

\begin{appendices}

    \section{Assumptions: Proofs}\label{app:assumptions-proofs}

    \input{appendices/assumptions_proofs}

    \section{Convergence of Stochastic AGD: Proofs}\label{app:convergence-proofs}

    \input{appendices/convergence_proofs}

    \section{Comparison to Existing Rates: Proofs}\label{app:comaprison-proofs}

    \input{appendices/comparison_proofs}

    \section{\textcolor{red}{Theoretical Issues in the Preprint}}\label{app:bug}

    \input{appendices/bug}

\end{appendices}


\end{document}

%% file: macros/math.tex
\def\bbE{\mathbb{E}}

\def\bbN{\mathbb{N}}

\def\calP{\mathcal{P}}

\def\calW{\mathcal{W}}

\usepackage{amsthm}
\usepackage{thmtools, thm-restate}

\usepackage{mathtools} % required for \DeclarePairedDelimeter

\newcommand{\rbr}[1]{\left(#1\right)}
\newcommand{\sbr}[1]{\left[#1\right]}
\newcommand{\cbr}[1]{\left\{#1\right\}}
\newcommand{\abr}[1]{\left\langle#1\right\rangle}

\def\norm#1{\|#1\|}

\def\half{\frac 1 2}

\newcommand{\R}{\mathbb{R}}

\newcommand{\into}{\rightarrow}

\newcommand{\Lmax}{L_{\text{max}}}

\newcommand{\yk}{y_k}
\newcommand{\ykk}{y_{k+1}}

\newcommand{\vk}{v_k}

\newcommand{\vkk}{v_{k+1}}

\newcommand{\w}{w}
\newcommand{\wk}{w_k}
\newcommand{\wkk}{w_{k+1}}
\newcommand{\wopt}{w^*}

\newcommand{\xk}{x_k}

\newcommand{\xbar}{\xbar{w}}

\newcommand{\z}{z}
\newcommand{\zk}{z_{k}}

\newcommand{\etamin}{\eta_{\text{min}}}

\newcommand{\etak}{\eta_k}

\newcommand{\Ek}{\bbE_{\zk}}
\newcommand{\E}{\bbE}

\newcommand{\f}{f}

\newcommand{\grad}{\nabla f}
\newcommand{\sgc}{\rho}

\newcommand{\fls}{f_{\text{ls}}}
\newcommand{\gradls}{\nabla f_{\text{ls}}}

%% file: sections/intro.tex
%!TEX root=../main.tex

A continuing trend in machine learning is the adoption of powerful prediction
models which can exactly fit, or \emph{interpolate}, their training data
\citep{zhang2017understanding}.
Methods such as over-parameterized neural networks \citep{zhang2013gradient,
    belkin2019reconciling}, kernel machines \citep{belkin2019datainterp}, and
boosting \citep{schapire1997boosting} have all been shown to achieve zero
training loss in practice.
This phenomena is particularly prevalent in modern deep learning, where
interpolation is conjectured to be key to both optimization
\citep{liu2022loss, oymak2019over} and generalization \citep{belkin21fit}.

Recent experimental and theoretical evidence shows stochastic gradient descent
(SGD) matches the fast convergence rates of deterministic gradient methods up
to problem-dependent constants when training interpolating models
\citep{arora2018overparameterization, ma2018power, zhou2019analysis}.
With additional assumptions, interpolation also implies the strong
\citep{polyak1987introduction} and weak \citep{bassily2018exponential,
    vaswani2019painless} growth conditions, which bound the second moment of the
stochastic gradients.
Under strong/weak growth, variance-reduced algorithms typically exhibit slower
convergence than stochastic gradient methods despite using more computation or
memory \citep{defazio2019effectiveness, ma2018power}, perhaps because these
conditions already imply a form of ``automatic variance reduction''
\citep{liu2022loss}.
A combination of interpolation and growth conditions has been used to prove
fast convergence rates for SGD with line-search \citep{vaswani2019painless},
with the stochastic Polyak step-size \citep{loizou2020sps, berrada2020sps},
for mirror descent \citep{dorazio2021smd}, and for model-based methods
\citep{asi2019stochastic}.

While these results show interpolation is sufficient to break the \(
\Omega(\epsilon^{-4}) \) complexity barrier for computing stationary points of
smooth, convex functions with stochastic, first-order oracles
\citep{arjevani2019lower}, significantly less work has been done to obtain the
accelerated rates possible in the deterministic setting
\citep{nemirovsky1985optimal}.
\citet{vaswani2019fast} analyze a stochastic version of Nesterov's accelerated
gradient method (AGD) \citep{nesterov1983method} under the strong growth
condition, but their bounds have a linear dependence on the strong growth
constant and can be slower than SGD \citep{liu2020accelerating}.
In contrast, \citet{liu2020accelerating} propose a modified version of
stochastic AGD and extend the statistical condition number approach of
\citet{jain2018accelerating} to the interpolation setting.
However, their results apply primarily to quadratics and are not accelerated
for general convex functions.

In this work, we apply the estimating sequences analysis developed by
\citet{nesterov2004lectures} to the interpolation setting.
Our approach hinges on a simple, in-expectation progress guarantee for SGD,
which we prove is a sufficient condition for generic acceleration of
stochastic algorithms.
This proof technique is completely different from that used by
\citet{vaswani2019fast} and yields an improved dependence on the strong growth
constant.
In the worst-case, the improvement is proportional to the square-root of the
condition number and guarantees stochastic AGD is always at least as fast as
SGD.
\
In what follows, all proofs are deferred to the corresponding
section of the appendix.

\subsection{Additional Related Work}

A large literature on stochastic optimization under interpolation rapidly
developed following the seminal work by \citet{bassily2018exponential} and
\citet{vaswani2019fast}.
For instance, \citet{xiao2022conditional} analyze Frank-Wolfe under
interpolation, \citet{vaswani2020adaptive} prove fast convergence for
Adagrad-type methods \citep{duchi2011adagrad}, and
\citet{meng2020fastandfurious} show fast rates for sub-sampled Newton method
under interpolation.
Interpolation has also been used to study last-iterate convergence of SGD
\citep{varre2021last} and subgradient methods \citep{fang2021subgradient}.

%Interpolation also plays a large role in developing methods with adaptive
%step-sizes.
%In particular, the Polyak step-size \citep{polyak1987introduction} has been
%extensively studied under interpolation conditions \citep{loizou2020sps,
%    berrada2020sps}, with follow-up works considering more general stochastic
%models \citep{orvieto2022sps}, mirror-descent \citep{dorazio2021smd},
%connections to passive-aggressive methods \citep{gower2022slack}, and
%modifications to avoid knowledge of the optimal value \citep{orvieto2022sps,
%    horvath2022adaptive}.
%Recently, second-order variants have also been developed \citep{li2023sp2}.

%Finally, another line of work has considered model-based methods, such as aProx
%\citep{asi2019importance}.
%These algorithms obtain linear and fast sub-linear rates under
%interpolation \citep{asi2020minibatch} as well as under the weak growth
%condition \citep{chadha2022model}.
%Although \citet{chadha2022model} complement their results on aProx with
%lower bounds for the interpolation setting, these rely on a $\gamma$-Growth
%condition whose constant is difficult to compute.

Interpolation is a key sufficient condition for growth conditions, which are a
set of general assumptions controlling the magnitude of the noise in
stochastic gradients.
Strong growth was introduced by \citet[Section 4.2.5]{polyak1987introduction},
who called the condition \emph{relative random noise} and used it to prove
q-linear convergence of SGD.
\citet{solodov1998incremental} and \citet{tseng1998incremental} later used a
variation of strong growth to analyze incremental gradient methods; their
variation was also used by \citet{schmidt2013fast} to prove linear
convergence of SGD with a constant step-size for strongly-convex functions.
\citet{vaswani2019fast} returned to the original definition given by
\citet{polyak1987introduction} and used it to analyze stochastic AGD, as we
do in this paper.

%\citet{vaswani2019fast} also proposed a relaxation of strong growth, called the
%weak growth condition, which allowed them to derive faster rates for
%non-strongly convex optimization.
%\citet{cevher2018linear} proposed an alternative relaxation, also called weak
%growth, which does not imply interpolation;
%this version of weak growth has been used frequently (see e.g.
%\citet{bertsekas2000gradient}) and is related to the expected smoothness
%\citep{qian2019improvedrates, khaled2023better} and expected residual
%\citep{gower2021structured} assumptions.

Many authors have tried to accelerate stochastic gradient methods, including
under a variety of growth conditions.
By accelerate, we mean improve the dependence on the condition number or
improve the order of convergence (i.e.\ from \( O(1/k) \) to \( O(1/k^2) \))
as compared to SGD.
For example, \citet{schmidt2011convergence} establish orders of growth on the
gradient noise which still permit stochastic proximal-gradient methods to be
accelerated.
In contrast, \citet{aspremont2008approximate} and \citet{devolder2014first}
assume bounded, deterministic gradient errors to derive convergence rates,
while \citet{cohen2018acceleration} develop a noise-resistant acceleration
scheme.
%Other work shows that high-probability convergence is impossible for biased
%stochastic gradients \citep{honorio2012biased}.
Most recently, \citet{chen2020convergence} analyze stochastic AGD under
expected smoothness, but their rate only holds under interpolation when the
strong growth constant is less than two.

As discussed above, \citet{liu2020accelerating} extend the approach of
\citet{jain2018accelerating} to the interpolation setting.
Their assumptions imply strong growth and the analysis is limited to
least-squares problems, although similar rates have been obtained for
continuized AGD applied to kernel least-squares \citep{even2021continuized}.
\citet{valls2022accelerated} take a different view and extend the work by
\citet{vaswani2019fast} to constrained optimization.
Unfortunately, none of these algorithms are necessarily accelerated and both
\citet{assran2020convergence} and \citet{liu2020accelerating} prove that
stochastic AGD may not obtain accelerated rates of convergence even under
interpolation.
We address this criticism and make detailed comparisons between convergence
rates in \cref{sec:discussion}.

%% file: sections/assumptions.tex
%!TEX root=../main.tex

Consider the unconstrained minimization of a smooth, convex function \( f :
\R^d \into \R \).
We assume \( f \) has at least one minimizer \( \wopt \) and denote the
optimal set by \( \calW^* \).
At each iteration \( k \), we sample a stochastic gradient \( \grad(\wk, \zk)
\) such that, \[ \E_{\zk} \sbr{\grad(\wk, \zk)} = \grad(\wk), \] meaning the
stochastic gradient is unbiased.
We assume that \( \zk, \z_j \) are independent for \( k \neq j \), but they do
not need to have the same distribution.
The stochastic gradients satisfy the strong growth condition when
there exists \( \rho \geq 1 \) for which
\begin{equation}
    \label{eq:strong-growth}
    \E_{\zk} \sbr{\norm{\grad(\wk, \zk)}_2^2} \leq \rho \norm{\grad(\wk)}_2^2,
\end{equation}
holds given any sequence \( \cbr{\wk, \zk} \).
We say that interpolation is satisfied if \[ \nabla f(\w) = 0 \implies \nabla
    f(\w, \zk) = 0, \] for all \( \zk \).
That is, stationarity of \( f \) implies stationarity of the stochastic
gradients.
Although the strong growth condition implies interpolation, we will see that
the converse requires further assumptions on \( f \) and on the stochastic
gradients.

We assume that \( f \) is \( L \)-smooth, meaning \( \nabla f \) is \( L
\)-Lipschitz continuous.
\( L \)-smoothness of \( f \) implies the following
quadratic upper-bound for all \( w, u \in \R^d \) \citep{nesterov2004lectures}:
\begin{equation}
    \label{eq:smooth}
    f(u) \leq f(w) + \abr{\grad(w), u - w} + \frac{L}{2} \norm{u - w}_2^2.
\end{equation}
Similarly, we will sometimes require the stochastic gradients to be \( \Lmax \)-individually
smooth, meaning they satisfy
\begin{equation}
    \label{eq:individually-smooth}
    f(u, \zk) \leq f(w, \zk) + \abr{\grad(w, \zk), u - w}
    + \frac{\Lmax}{2} \norm{u - w}_2^2,
\end{equation}
almost surely for all \( k \).
We also assume that \( f \) is \( \mu \)-strongly convex,
by which we mean
\begin{equation}
    \label{eq:strongly-convex}
    f(u) \geq f(w) + \abr{\grad(w), u - w} + \frac{\mu}{2} \norm{u - w}_2^2,
\end{equation}
holds for all \( w, u \in \R^d \) and some \( \mu \geq 0 \).
When \( \mu = 0 \), strong convexity reduces to convexity of \( f \).
If \( f \) is strongly convex with \( \mu > 0 \), the stochastic gradients are
\( \Lmax \)-individually smooth, and interpolation holds, then the strong
growth constant is bounded as \( \rho \leq \Lmax / \mu \) (see
\cref{lemma:strong-growth-constant}).
Recalling that the ratio \( \kappa = L / \mu \) is the condition number of \(
f \), we see that the strong-growth constant is bounded by a quantity like the
condition number of the worst-conditioned stochastic function.
\footnote{Note that \( f(u, \zk) \) is typically not strongly convex, so this analogy is not formal.}

%% file: sections/convergence.tex
%!TEX root=../main.tex
Our analysis in this section builds on the estimating sequences approach
developed by \citet{nesterov1988approach}.
However, we consider variable step-sizes \( \etak \) and allow the iterates \(
\wk \) to be set using a general update scheme.
The procedure takes \( \gamma_0 > 0 \) as input and uses the following updates,
which we call \emph{generalized stochastic AGD}:
\begin{equation}
    \label{eq:stochastic-agd}
    \begin{aligned}
        \alpha_{k}^2
             & = \etak (1 - \alpha_{k}) \gamma_k
        + \eta_{k} \alpha_{k} \mu                                             \\
        \gamma_{k+1}
             & = (1 - \alpha_k) \gamma_{k} + \alpha_k \mu                     \\
        \yk
             & = \frac{1}{\gamma_k + \alpha_k \mu} \sbr{\alpha_k \gamma_k \vk
        + \gamma_{k+1} \wk},                                                  \\
        \wkk
             & = m(\etak, \yk, \grad(\yk, \zk))                               \\
        \vkk & = \frac{1}{\gamma_{k+1}} \sbr{(1 - \alpha_k) \gamma_k \vk
            + \alpha_k \mu \yk - \alpha_k \grad(\yk)}.
    \end{aligned}
\end{equation}
where \( m \) is an as-yet unspecified update for the ``primal sequence''
\( \wk \) and \( v_0 = \w_0 \) (which implies \( y_0 = w_0 \)).
Note that the step-size \( \etak \) is required at the start of the iteration
to compute \( \alpha_k \).
As a result, local search methods like the Armijo line-search
\citep{armijo1966ls} must re-evaluate \( \gamma_{k+1}, \yk \), and \( \wkk \)
for each candidate step-size.

Choosing \( m \) to be one step of SGD with step-size \( \etak \) yields the
familiar updates of the standard version of stochastic AGD (see \citet[Eq.
    2.2.20]{nesterov2004lectures}).
(\textcolor{red}{\textbf{Warning}: The preceding claim is not true.  The use of a
    deterministic gradient in the update for \( \vkk \) in \cref{eq:stochastic-agd}
    breaks the standard argument that this scheme is equivalent to the ``momentum''
    form of stochastic AGD in \cref{eq:stochastic-agd-simple}.
    Please see \cref{app:bug} for more details.})
\begin{equation}
    \label{eq:stochastic-agd-simple}
    \begin{aligned}
        \wkk
         & = \wk - \etak \grad(\yk, \zk)                            \\
        \alpha_{k+1}^2
         & = (1 - \alpha_{k+1}) \alpha_k^2 \frac{\eta_{k+1}}{\etak}
        + \eta_{k+1} \alpha_{k+1} \mu                               \\
        \ykk
         & = \wkk
        + \frac{\alpha_k (1 - \alpha_k)}{\alpha_k^2 + \alpha_{k+1}} (\wkk - \w_k).
    \end{aligned}
\end{equation}
Our approach hinges on the fact that, under strong growth, stochastic gradient
updates for \( \wk \) give a similar per-step progress condition as
deterministic gradient descent.
Since descent in the primal step is the only link between \( \wk \) and the
``dual sequence'' \( \yk \), generalized stochastic AGD with any primal update
obtaining similar per-iteration progress can be analyzed in the same fashion
as stochastic AGD.
This allows us to derive a fast convergence rate for the general scheme in
\cref{eq:stochastic-agd}.

We start by deriving the progress condition for SGD.
\
It is straightforward to prove the following bound using \( L \)-smoothness,
the definition of \( \wkk \), and strong growth:
\begin{restatable}{lemma}{progress}\label{lemma:descent-lemma}
    Suppose \( f \) is \( L \)-smooth, the strong growth condition holds, and \(
    \etak \) is independent of \( \zk \).
    Then the stochastic gradient step in \cref{eq:stochastic-agd-simple} makes
    progress as,
    \begin{equation}
        \label{eq:descent-lemma}
        \E_{\zk}\sbr{f(\wkk)} \leq f(\yk) - \etak(1 - \frac{\etak \rho L}{2})\norm{\grad(\yk)}_2^2.
    \end{equation}
\end{restatable}
Substituting any fixed step-size \( \etak \leq 1 / \rho L \) into
\cref{eq:descent-lemma} gives the following equation, which we call
the \emph{expected progress condition}:
\begin{equation}
    \label{eq:progress}
    \E_{\zk}\sbr{f(\wkk)} \leq f(\yk) - \frac{\etak}{2} \norm{\grad(\yk)}_2^2,
\end{equation}
which is equivalent to the progress made by gradient descent with
step-size \( \etak \leq 1 / L \) up to a factor of \( \rho \)
\citep{bertsekas1997nonlinear}.
In order to make use of the expected progress condition, we now introduce the
estimating sequences framework.
\begin{definition}[Estimating Sequences]
    Two sequences \( \lambda_k \), \( \phi_k \) are estimating sequences
    if the following hold almost surely:
    (i) \( \lambda_k \geq 0 \)
    (ii) \( \lim \lambda_k = 0 \),
    and (iii) for all \( w \in \R^d \),
    \begin{equation}
        \label{eq:estimating-sequence}
        \phi_k(w)
        \leq (1 - \lambda_k) f(w) + \lambda_k \phi_0(w).
    \end{equation}
\end{definition}
Unlike the standard definition, we permit \( \lambda_k \) and \( \phi_k \) to
depend on \( z_0, \ldots, z_{k-1} \), making them random variables.
Typically the first function \( \phi_0 \) is deterministic and chosen to
satisfy \( \phi_0(w) \geq f(w) \) for all \( w \) near \( w_0 \).
Since \cref{eq:estimating-sequence} guarantees \( \lim_k \phi_k(w) \leq f(w)
\), \( \phi_k \) can be interpreted as a sequence of relaxing upper-bounds,
where the rate of relaxation is controlled by \( \lambda_k \).
The next condition captures when \( \phi_k \) is a good local model of \( f
\).
\begin{definition}
    \label{def:local-upper-bound}
    The local upper-bound property holds in expectation if
    \begin{equation}
        \label{eq:local-upper-bound}
        \E_{\z_0, \ldots, \z_{k-1}}\sbr{\f(\wk)} \leq \E_{\z_0, \ldots, \z_{k-1}}\sbr{\inf_u \phi_k(u)}.
    \end{equation}
\end{definition}
In what follows, we use \( \E \) without subscripts to denote the total
expectation with respect to \( \z_0, \ldots, \z_{k-1} \).
That is, all randomness in the procedure up to iteration \( k \).
If \( \phi_k \) maintains the local upper-bound property in expectation,
then \cref{eq:estimating-sequence} guarantees
\begin{align}
    \E\sbr{f(\wk)}
    \leq \E\sbr{\inf_{u} \phi_k(u)}
     & \leq \E\sbr{\phi_k(\wopt)}
    \leq \E\sbr{(1 - \lambda_k) f(\wopt) + \lambda_k \phi_0(\wopt)} \nonumber              \\
    \implies \E\sbr{f(\wk)} - f(\wopt)
     & \leq \E\sbr{\lambda_k(\phi_0(\wopt) - f(\wopt))}, \label{eq:estimating-convergence}
\end{align}
which shows that generalized stochastic AGD converges in expectation at a rate
controlled by \( \lambda_k \).
As a result, the bulk of our analysis focuses on establishing the local
upper-bound property for a suitable choice of estimating sequences.

Following \citet{nesterov2004lectures}, choose \( \lambda_0 = 1\),
\( \phi_0(w) = f(w_0) + \frac{\gamma_0}{2}\norm{w - w_0}_2^2 \), and
\begin{equation}
    \label{eq:sequence-choice}
    \begin{aligned}
        \lambda_{k+1}
         & = (1 - \alpha_k) \lambda_k \\
        \phi_{k+1}(w)
         & = (1 - \alpha_k) \phi_k(w)
        + \alpha_k \rbr{f(\yk) + \abr{\grad(\yk), w - \yk}
            + \frac{\mu}{2}\norm{w - \yk}_2^2}.
    \end{aligned}
\end{equation}
The initial curvature \( \gamma_0 \) is an input parameter; differentiating
shows that \( \vkk \) is actually the minimizer of \( \phi_{k+1} \), while \(
\nabla^2 \phi_{k+1} = \gamma_{k+1} I \) (\cref{lemma:canonical-form}).
Thus, the auxillary sequences \( \gamma_k, \vk \) can be viewed as arising
from our choice of local model.
The next lemma proves these are valid estimating sequences when the step-size
sequence is well-behaved.
In what follows, we use the convention \( 1 / 0 = \infty \) to cover the case
of non-strongly convex functions.
\begin{restatable}{lemma}{sequenceChoice}\label{lemma:sequence-choice}
    Assume \( f \) is \( \mu \)-strongly convex with \( \mu \geq 0 \) and \(
    \etamin \leq \etak < 1 / \mu \) almost surely.
    Then \( \lambda_k \) and \( \phi_k \) given in \cref{eq:estimating-sequence}
    are estimating sequences.
\end{restatable}
The parameter \( \etamin > 0 \) is required for \( \lambda_k \) to decrease
sufficiently fast, while the upper-bound \( \etak \leq 1 / \mu \) is only
necessary when \( \mu > 0 \).
In this case, it guarantees \( \lambda_k \geq 0 \).
This choice of estimating sequences also satisfies the local error-bound
property in expectation when \( m(\etak, \yk, \grad(\yk, \zk)) \) matches the
progress of fixed step-size SGD.
\begin{restatable}{proposition}{localUpperBound}\label{prop:local-upper-bound}
    If \( f \) is \( L \)-smooth and \( \mu \)-strongly convex with \( \mu \geq 0
    \), \( \etamin \leq \etak < 1 / \mu \) almost surely for every \( k \in \bbN
    \), and the primal update \( m(\etak, \yk, \grad(\yk, \zk)) \) satisfies the
    sufficient progress condition (\cref{eq:progress}), then \( \phi_k \) has the
    local upper-bound property in expectation.
    That is, for every \( k \in \bbN \),
    \begin{equation*}
        \E\sbr{f(\wk)} \leq \E\sbr{\inf_u \phi_k(u)}.
    \end{equation*}
\end{restatable}
\cref{prop:local-upper-bound} is our main theoretical contribution and immediately
leads to two accelerated convergence rates for the generalized stochastic AGD
scheme.
\begin{restatable}{theorem}{linearConvergence}\label{thm:linear-convergence}
    (\textcolor{red}{\textbf{Warning}: This result only holds for the
        scheme in \cref{eq:stochastic-agd}.
        It does not hold for stochastic AGD.})
    Suppose \( f \) is \( L \)-smooth and \( \mu \)-strongly convex with \( \mu >
    0 \), \( \etamin \leq \etak < 1 / \mu \) almost surely for every \( k \in \bbN
    \), and the primal update \( m(\etak, \yk, \grad(\yk, \zk)) \) satisfies the
    sufficient progress condition (\cref{eq:progress}).
    If \( \gamma_0 = \mu \), then generalized stochastic AGD has the following
    rate of convergence:
    \begin{equation}
        \label{eq:linear-rate}
        \begin{aligned}
            \E\sbr{f(\wkk)} - f(\wopt)
             & \leq \E\sbr{\prod_{i=0}^k \rbr{1 - \sqrt{\etak \mu}}}
            \sbr{f(\w_0) - f(\wopt) + \frac{\mu}{2}\norm{\w_0 - \wopt}_2^2} \\
             & \leq \big(1 - \sqrt{\etamin \mu}\big)^{k+1}
            \sbr{f(\w_0) - f(\wopt) + \frac{\mu}{2}\norm{\w_0 - \wopt}_2^2}.
        \end{aligned}
    \end{equation}
\end{restatable}

This linear rate of convergence requires knowledge of the strong convexity
constant in order to set \( \gamma_0 \).
However, we can still obtain a \( O(1/k^2) \) rate without knowing \( \mu \)
so long as the smoothness constant can be estimated.
The following theorem is a generalization of \citet[]{nesterov2004lectures} to
stochastic optimization under interpolation.

\begin{restatable}{theorem}{sublinearConvergence}\label{thm:sublinear-convergence}
    (\textcolor{red}{\textbf{Warning}: This result only holds for the
        scheme in \cref{eq:stochastic-agd}.
        It does not hold for stochastic AGD.})
    Suppose \( f \) is \( L \)-smooth and \( \mu \)-strongly convex with \( \mu
    \geq 0 \), \( \etamin \leq \etak < 1 / \mu \) almost surely for every \( k \in
    \bbN \), and the primal update \( m(\etak, \yk, \grad(\yk, \zk)) \) satisfies
    the sufficient progress condition (\cref{eq:progress}).
    If \( \gamma_0 \in (\mu, 3/\etamin) \), then generalized stochastic AGD has
    the following rate of convergence:
    \begin{equation}
        \label{eq:sublinear-rate}
        \E\sbr{f(\wkk)} - f(\wopt)
        \leq \frac{4}{\etamin (\gamma_0 - \mu)(k+1)^2}
        \sbr{f(\w_0) - f(\wopt) + \frac{\gamma_0}{2}\norm{\w_0 - \wopt}_2^2}.
    \end{equation}
\end{restatable}

\subsection{Specializations}\label{sec:specializations}

Theorems~\ref{thm:linear-convergence} and~\ref{thm:sublinear-convergence}
provide accelerated guarantees for any stochastic primal update \( m(\etak,
\yk, \grad(\yk, \zk)) \) satisfying the sufficient progress condition.
Assuming strong growth holds, we may specialize \( m \) to fixed step-size SGD
with \( \etak = 1 / \rho L \) (sufficient progress is satisfied according to
\cref{lemma:descent-lemma}).
This yields the standard version of stochastic AGD analyzed by
\citet{vaswani2019fast}.
However, instantiating our convergence guarantees shows a faster rate for
stochastic AGD with an improved dependence on \( \rho \).

\begin{corollary}%
    \label{cor:stochastic-agd-convergence}
    (\textcolor{red}{\textbf{Warning}: This result only holds for the
        scheme in \cref{eq:stochastic-agd}.
        It does not hold for stochastic AGD.})
    If \( f \) is \( L \)-smooth and \( \mu \)-strongly convex with \( \mu > 0 \),
    strong growth holds, and \( \etak = 1/\rho L \), then
    stochastic AGD with \( \gamma_0 = \mu \) converges as,
    \begin{equation}
        \E\sbr{f(\wkk)} - f(\wopt)
        \leq \rbr{1 - \sqrt{\frac{\mu}{\rho L}}}^{k+1}
        \sbr{f(\w_0) - f(\wopt) + \frac{\mu}{2}\norm{\w_0 - \wopt}_2^2}.
    \end{equation}
    Alternatively, if \( \mu \geq 0 \) and \( \gamma_0 \in (\mu, 3\rho L) \),
    then stochastic AGD satisfies,
    \begin{equation}
        f(\wkk) - f(\wopt)
        \leq \frac{4 \rho L}{(\gamma_0 - \mu)(k+1)^2}
        \sbr{f(\w_0) - f(\wopt) + \frac{\gamma_0}{2}\norm{\w_0 - \wopt}_2^2}.
    \end{equation}
\end{corollary}

\cref{cor:stochastic-agd-convergence} shows that SGD can be accelerated to
obtain the same convergence rate as deterministic AGD up to a factor
of \( \rho \).
We emphasize that some dependence on \( \rho \) cannot be avoided; generic
acceleration of SGD is not possible, even in the interpolation setting
\citep{assran2020convergence}, so convergence rates must incorporate some
measure of hardness due to stochasticity.
In the next section, we compare our convergence guarantees against other
results from the literature and give simple conditions under which
acceleration is achieved.

An advantage of our analysis is that it also extends to more
complex methods, such as SGD with full matrix preconditioning,
\begin{equation}
    \label{eq:preconditioned SGD}
    \wkk = \wk - \etak D_k^{-1} \grad(\wk, \zk),
\end{equation}
where \( D_k \in \R^{d \times d} \) is a positive-definite matrix and \( \etak \)
is a step-size sequence.
We say that matrix strong growth holds in the norm \( \norm{x}^2_{D_k} = x^\top D_{k} x \) with constant \( \rho_{D_{k}} \) if
\begin{equation}
    \label{eq:matrix-strong-growth}
    \Ek \sbr{\norm{\grad(\w,\zk)}_{D_{k}^{-1}}^2}
    \leq \rho_{D_k} \norm{\grad(w)}_{D_{k}^{-1}}^2.
\end{equation}
If interpolation holds, \( \grad(u, \zk) \) are \( \Lmax^{D_{k}}
\)-individually smooth in \( \norm{\cdot}_{D_k} \), and \( f \) is \(
\mu_{D_k} \)-strongly convex in \( \norm{\cdot}_{D_k} \), then \( \rho_{D_k}
\leq \Lmax^{D_k} / \mu_{D_k} \) (see \cref{lemma:matrix-strong-growth}) and
the following convergence rate is obtained by combining
\cref{lemma:matrix-descent-lemma} with our main convergence theorems for
generalized stochastic AGD.
\begin{corollary}
    \label{cor:preconditioned-convergence}
    Assume \( f \) is \( \mu \)-strongly convex and \( 0 \prec D_k \preceq I \)
    for every \( k \in \bbN \).
    Suppose \( f \) is \( L_{D_k} \)-smooth, \( \rho_{D_k} \) matrix strong growth
    holds, and \( \etak = 1 / \rho_{D_k} L_{D_k} \).
    Let \( C_{\infty} = \sup_{k} \cbr{\rho_{D_k} L_{D_k}} \).
    If \( \gamma_0 = \mu > 0 \), then stochastic preconditioned AGD
    satisfies,
    \begin{equation}
        \E\sbr{f(\wkk)} - f(\wopt)
        \leq \prod_{i=0}^{k} \rbr{1 - \sqrt{\frac{\mu}{\rho_{D_k}
                    L_{D_k}}}} \sbr{f(\w_0) - f(\wopt) + \frac{\mu}{2}\norm{\w_0 - \wopt}_2^2}.
    \end{equation}
    Alternatively, if \( \mu \geq 0 \)
    and \( \gamma_0 \in (\mu, 3 C_\infty) \), then we obtain,
    \begin{equation}
        f(\wkk) - f(\wopt)
        \leq \frac{4 C_\infty}{(\gamma_0 - \mu)(k+1)^2}
        \sbr{f(\w_0) - f(\wopt) + \frac{\gamma_0}{2}\norm{\w_0 - \wopt}_2^2}.
    \end{equation}
\end{corollary}
Compared to standard stochastic AGD, preconditioning allows us to measure
stochasticity in the norm induced by \( D_k \).  This is advantageous when \(
\rho_{D_K} L_{D_k} \leq \rho L \), meaning \( f \) is smoother and/or the
stochastic gradients are better conditioned in \( \norm{\cdot}_{D_k} \) than in
the standard Euclidean norm.  In such a setting, our theory suggests
preconditioning is a simple way to further speed-up accelerated stochastic
optimization.

%It is also straightforward to modify \cref{cor:preconditioned-convergence} to
%permit a sequence of changing preconditioners, \( \cbr{D_k} \).
%If we choose adaptive step-sizes \( \etak = 1 / \rho_{D_k} L_{D_k} \) and set
%\( \gamma_0 = \mu > 0 \), then \cref{thm:linear-convergence} immediately
%yields a convergence bound Furthermore, a similar result holds when \( \mu = 0
%\) or \( \gamma_0 > \mu \) via \cref{thm:sublinear-convergence}; in both
%cases, the final convergence rates are accelerated.

%% file: sections/discussion.tex
%!TEX root=../main.tex

\begin{table}[t]
    \centering
    \begin{tabular}{c c c c c}\toprule
        Assumptions%
               & SGD                                                                 %
               & S-AGD (Ours)
               & S-AGD (VSB)                                                         % 
               & MaSS                                                                %
        \\ \midrule
        Strongly Convex
               & \( O\rbr{\frac{\Lmax}{\mu} \log(\frac{1}{\epsilon})}  \)            %
               & \( O\rbr{\sqrt{\frac{\rho L}{\mu}} \log\rbr{\frac{1}{\epsilon}}} \) %                                            
               & \( O\rbr{\rho \sqrt{\frac{L}{\mu}} \log\rbr{\frac{1}{\epsilon}}} \) %                                          
               & \( \sqrt{\kappa \tilde \kappa } \log(\frac{1}{\epsilon}) \)         % 
        \\
        Convex & \( O\rbr{\frac{\Lmax}{\epsilon}} \)                                 %
               & \( O\rbr{\sqrt{\frac{\sgc L}{\epsilon}}\,} \)                       %
               & \( O\rbr{\rho \sqrt{\frac{L}{\epsilon}}\,} \)                       %                                                                
               & \text{N/A}                                                          %
        \\ \bottomrule
    \end{tabular}
    \caption{
        Comparison of iteration complexities for stochastic acceleration schemes under
        strong growth and individual smoothness.
        VSB indicates \citet{vaswani2019fast} and MaSS is the modified stochastic AGD
        iteration proposed by \citet{liu2020accelerating}.
        The strongly-convex rate for MaSS applies only to quadratics; although MaSS
        has a convergence guarantee for convex functions, we omit it here because it
        relies on a hard-to-interpret assumption and is not accelerated.
    }%
    \label{table:agd-comparison}
\end{table}

Now we compare our rates to those existing in the literature.
Throughout this section, we assume that \( f \) is individually smooth,
interpolation holds, and the strong growth condition is satisfied.
Recall that the strong growth constant is bounded above as \( \rho \leq \Lmax
/ \mu \) under these conditions.
This worst-case bound on \( \rho \) is critical to understanding when
stochastic AGD does or does not accelerate.

Before proceeding, we introduce the notion of statistical condition number
proposed by \citet{jain2018accelerating} and used by
\citet{liu2020accelerating} to analyze their modified version of stochastic
AGD (called MaSS) in the least-squares setting.
Let \( \calP \) be a probability distribution over \( (x, y) \) and define
the least squares objective as
\begin{equation}
    \label{eq:least-squares}
    f_{\text{ls}}(w) = \E_{(x, y) \sim \calP} \sbr{\half (x^\top w - y)^2}.
\end{equation}
Define the stochastic functions and gradients to be \( f_{\text{ls}}(w, \zk) =
\half (x_{\zk}^\top w - y_{\zk})^2 \) and \( \grad_{\text{ls}}(w, \zk) =
x_{\zk} (x_{\zk}^\top w - y_{\zk}), \) where \( (x_{\zk}, y_{\zk}) \sim \calP
\).
These stochastic gradients are \( \Lmax \)-individually smooth with \( \Lmax =
\sup \cbr{ \norm{x_{\zk}}_2^2 : (x_{\zk}, y_{\zk}) \in \text{Supp}(\calP)} \).
Assuming we may interchange expectation and differentiation, the Hessian is \(
H = \E_x \sbr{x x^\top} \)
and the condition number \( \kappa \) and
statistical condition number \( \tilde \kappa \) are defined as,
\begin{equation}
    \begin{aligned}
        \kappa
        = \inf\cbr{ t / \mu : \E_x\sbr{\norm{x}_2^2 (x x^\top)} \preceq t H },
        \hspace{0.5em}
        \tilde \kappa
        = \inf\cbr{ t : \E_x\sbr{\norm{x}^2_{H^{-1}} (x x^\top)} \preceq t H }.
    \end{aligned}
\end{equation}
It is straightforward to prove \( \kappa, \tilde \kappa \leq \Lmax / \mu \),
similar to the strong growth constant.

\cref{table:agd-comparison} compares our iteration complexities for stochastic AGD to the
complexity of SGD under interpolation, the analysis of stochastic AGD by
\citet{vaswani2019fast}, and the complexity of MaSS.
Unlike \citet{vaswani2019fast}, who use strong growth to show both the
optimality gap and distance to a minimizer decrease in expectation at each
iteration, our approach only requires the sufficient progress condition.
This allows us to shrink the dependence on the strong growth constant from \(
\rho \) to \( \sqrt{\rho} \), which --- since \( \rho \leq \Lmax / \mu \) ---
can be larger than \( \sqrt{\kappa} \) in the worst case.
Substituting this into the complexity bound shows stochastic AGD requires \(
O((\sqrt{L \Lmax} / \mu) \log(1 / \epsilon)) \) iterations to reach \(
\epsilon \)-sub-optimality.
That is, stochastic AGD is always at least as fast SGD and faster when \(
\Lmax \gg L \).

Our convergence rate for stochastic AGD also improves over that for SGD under
the strong growth condition \citep{schmidt2013fast}.
The improvement is by a factor \( \sqrt{\rho / \mu} \), indicating that
acceleration actually shrinks the dependence on the noise level.
This quite different from results in the general stochastic setting, where
accelerated methods are typically more sensitive to noise
\citep{honorio2012biased}.
For example, \citet{schmidt2011convergence} show that the noise level must
decrease faster for accelerated methods to converge compared to (proximal) SGD
when interpolation does not hold.
We conclude that interpolation seems to be key when proving fast rates for
stochastic AGD.

Comparing our results against MaSS is more difficult due to the dependence on
\( \tilde \kappa \).
To understand the difference in convergence rates, we consider two finite-sum
example problems.
In what follows, let \( e_1, \ldots, e_n \) be the standard basis for \( \R^n
\).
\begin{restatable}[\( \Lmax \gg L \)]{example}{uniformBasis}\label{ex:uniform-basis}
    Consider the least-squares problem setting in \cref{eq:least-squares} and
    choose \( y = 0 \), \( x \sim \text{Uniform}(e_1, \ldots e_n) \).
    A short calculation shows \( \Lmax = 1 \), \( \rho = n \), \( L = \mu = 1 / n
    \), and \( \kappa = \tilde \kappa = n \).
    As a result, stochastic AGD and MaSS have the following complexity bounds:
    \begin{equation}
        \text{\emph{S-AGD: }}
        O\rbr{\sqrt{n} \log\rbr{\frac{1}{\epsilon}}} \quad \text{\emph{vs}} \quad
        \text{\emph{MaSS: }} O\rbr{n\log\rbr{\frac{1}{\epsilon}}}
    \end{equation}
\end{restatable}

As expected, stochastic AGD accelerates
due to the gap between the smoothness and individual smoothness constants.
In comparison, the complexity bound for MaSS is not accelerated and only
matches that for SGD.
The next example considers the opposite setting, where \( L \approx \Lmax \)
and we do not expect stochastic AGD to be faster than SGD.

\begin{restatable}[\( \Lmax \approx L \)]{example}{nonUniformBasis}
    Consider the least-squares problem setting in \cref{eq:least-squares}.
    Let \( y = 0 \) and \( x \) be distributed as follows: \( P(x = e_1) = 1 - 1/n
    \) and \( P(x = e_2) = 1/n \).
    It is straightforward to show that \( \Lmax = 1 \), \( \mu = 1 / n \), and \(
    L = (n - 1)/n \), while \( \rho = n \) and \( \tilde \kappa = \kappa = n \).
    As a result, the complexity estimates for stochastic AGD and MaSS are,
    \begin{equation}
        \text{\emph{S-AGD: }}
        O\rbr{\sqrt{n(n-1)} \log\rbr{\frac{1}{\epsilon}}} \quad \text{\emph{vs}} \quad
        \text{\emph{MaSS: }} O\rbr{n\log\rbr{\frac{1}{\epsilon}}}.
    \end{equation}
\end{restatable}

As \( n \rightarrow \infty \), stochastic AGD attains the same complexity as
SGD and is not accelerated.
In comparison, the guarantee for MaSS always matches SGD and is slower than
stochastic AGD for every finite \( n \).
We conclude that while both methods are restricted by lower bounds on
stochastic acceleration, AGD can accelerate on some simple problems where MaSS
fails.

%% file: sections/conclusion.tex
%!TEX root=../main.tex

We derive new convergence rates for a generalized version of stochastic
Nesterov acceleration.
Our approach extends the estimating sequences framework to the stochastic
setting and shows that any update scheme making sufficient progress in
expectation can be accelerated.
As this sufficient progress condition is satisfied by SGD under the strong
growth condition, our proof immediately specializes to give fast
rates for stochastic AGD.
Compared to previous work, our convergence bounds improve the dependence on
the strong growth constant from \( \rho \) to \( \sqrt{\rho} \).
This improvement can be larger than the square-root of the condition number,
shows stochastic AGD is at least as fast as SGD, and explains the strong
empirical performance of stochastic acceleration shown by
\citet{vaswani2019fast}.
We also leverage our generalized algorithm to prove convergence guarantees for
stochastic AGD with preconditioning.
In particular, we show that preconditioning further speeds-up accelerated SGD
when the stochastic gradients are small in the matrix norm induced by the
preconditioner.

In addition to these results, the utility of our theoretical approach is
further demonstrated by recent literature.
Our core result for stochastic AGD (\cref{prop:local-upper-bound}) was
previously made available in a master's thesis
\citep{mishkin2020interpolation} and the proof technique has since been
leveraged to give optimal bounds for stochastic acceleration in the general
setting \citep{vaswani2022towards}.
Yet, several questions remain unanswered.
For example, the convergence of stochastic AGD under relaxed conditions, like
weak growth or with a stochastic line-search \citep{vaswani2019fast}, has not
been proved.
And while our generalized AGD scheme also suggests accelerating methods like
stochastic proximal-point, establishing the expected progress condition
appears difficult and new insights may be required.
We leave these questions to future work.

%% file: appendices/assumptions_proofs.tex
%!TEX root=../main.tex

\begin{lemma}
    \label{lemma:weak-growth-constant}
    Suppose \( f \) is convex and \( L \)-smooth, the stochastic gradients \(
    \cbr{\grad(\wk, \zk)} \) are \( \Lmax \) individually smooth, and
    interpolation holds.
    Then the weak growth condition holds with constant \( \alpha \leq \Lmax / L
    \).
\end{lemma}
\begin{proof}
    First, recall that the weak growth condition \citep{vaswani2019fast}
    is given by
    \begin{equation}
        \label{eq:weak-growth}
        \E_{\zk}\sbr{\norm{\grad(\wk, \zk)}_2^2}
        \leq 2 \alpha L \rbr{f(\wk) - f(\wopt)}.
    \end{equation}
    Now, starting from \( \Lmax \) individual-smoothness,
    \begin{align*}
        f(u, \zk)
         & \leq f(\w, \zk) + \abr{\grad(\w, \zk), u - \w}
        + \frac{\Lmax}{2}\norm{u - \w}^2,                                       \\
        \intertext{
            and choosing \( u = \w - \frac{1}{\Lmax}\grad(\w, \zk) \),
            we obtain
        }
        f(u, \zk)
         & \leq f(\w, \zk) - \frac{1}{\Lmax}\abr{\grad(\w, \zk), \grad(\w,\zk)}
        + \frac{\Lmax}{2\Lmax^2}\norm{\grad(\w,\zk)}^2                          \\
         & = f(\w, \zk) - \frac{1}{2\Lmax}\norm{\grad(\w,\zk)}^2.
    \end{align*}
    Noting that \( f(u,\zk) \geq f(\wopt, \zk) \) by convexity
    of \( f \) and interpolation and taking expectations with respect
    to \( \zk \) gives the following:
    \begin{align*}
        f(\wopt, \zk)
         & \leq f(\w, \zk) - \frac{1}{2\Lmax}\norm{\grad(\w,\zk)}^2                   \\
        \implies \Ek\sbr{f(\wopt, \zk)}
         & \leq \Ek\sbr{f(\w, \zk)} - \frac{1}{2\Lmax}\Ek\sbr{\norm{\grad(\w,\zk)}^2} \\
        \implies f(\wopt)
         & \leq f(\w) - \frac{1}{2\Lmax}\Ek\sbr{\norm{\grad(\w,\zk)}^2}.
    \end{align*}
    Re-arranging this final equation gives the desired result as follows:
    \begin{align*}
        \Ek \sbr{\norm{\grad(\w,\zk)}^2}
         & \leq 2 \Lmax \rbr{f(\w) - f(\wopt)} \\
         & = 2 \rbr{\frac{\Lmax}{L}}
        L \rbr{f(\w) - f(\wopt)}.
    \end{align*}
    We conclude that weak growth holds with \( \alpha \leq \frac{\Lmax}{L} \).
\end{proof}

\begin{lemma}
    \label{lemma:strong-growth-constant}
    Suppose \( f \) is \( L \)-smooth and \( \mu \)-strongly convex, the
    stochastic gradients \( \cbr{\grad(\wk, \zk)} \) are \( \Lmax \) individually
    smooth, and interpolation holds.
    Then strong growth holds with constant \( \rho \leq \Lmax / \mu \).
\end{lemma}
\begin{proof}
    \autoref{lemma:weak-growth-constant} implies that \( f \) satisfies the
    weak growth condition with parameter
    \[ \alpha \leq \frac{L_{\text{max}}}{L}. \]
    \citet[Proposition 1]{vaswani2019fast} now implies that \( f \)
    satisfies strong growth with parameter
    \[ \rho \leq \frac{\alpha L}{\mu} \leq \frac{L_{\text{max}}}{\mu}.  \]
    This concludes the proof.
\end{proof}

%% file: appendices/convergence_proofs.tex
%!TEX root=../main.tex

\progress*
\begin{proof}
    The proof is a modification of the standard descent lemma.
    Starting from \( L \)-smoothness of \( f \), we obtain
    \begin{align*}
        f(\wkk)
         & \leq f(\yk) + \abr{\grad(\yk), \wkk - \yk}
        + \frac{L}{2}\norm{\wkk - \yk}_2^2                                         \\
         & = f(\yk) - \etak \abr{\grad(\yk), \grad(\yk, \zk)}
        + \frac{\etak^2 L}{2} \norm{\grad(\yk, \zk)}_2^2                           \\
        \intertext{Taking expectations with respect to \( \zk \)
            and using the strong growth condition,}
        \E_{\zk} \sbr{f(\wkk)}
         & \leq f(\yk) - \etak \abr{\grad(\yk), \E_{\zk}\sbr{\grad(\yk, \zk)}}
        + \frac{\etak^2  L}{2} \E_{\zk}\sbr{\norm{\grad(\yk, \zk)}_2^2}            \\
         & \leq f(\yk) - \etak \norm{\grad(\yk)}_2^2
        + \frac{\etak^2  L \rho}{2} \norm{\grad(\yk)}_2^2                          \\
         & = f(\yk) - \etak\rbr{1 - \frac{\etak L \rho}{2}} \norm{\grad(\yk)}_2^2.
    \end{align*}

\end{proof}

\begin{lemma}
    \label{lemma:canonical-form}
    The \( \phi_k \) sequence in \cref{eq:sequence-choice}
    satisfies the following canonical form:
    \begin{equation}
        \phi_{k+1}(w) = \phi_{k+1}^* + \frac{\gamma_{k+1}}{2} \norm{w - \vkk}_2^2,
    \end{equation}
    where the curvature \( \gamma_{k+1} \), minimizer \( \vkk \), and minimum
    value \( \phi_{k+1}^* \) are given as follows:
    \begin{equation}
        \begin{aligned}
            \gamma_{k+1}
             & = (1 - \alpha_k) \gamma_k + \alpha_k \mu               \\
            \vkk
             & = \frac{1}{\gamma_{k+1}}
            \sbr{(1 - \alpha_k) \gamma_k \vk
            + \alpha_k \mu y_k  - \alpha_k \grad(\yk)}                \\
            \phi_{k+1}^*
             & = (1 - \alpha_k) \phi_k^* + \alpha_k f(\yk)
            - \frac{\alpha_k^2}{2 \gamma_{k+1}} \norm{\grad(\yk)}_2^2 \\
             & \hspace{2em}
            + \frac{\alpha_k (1 - \alpha_k) \gamma_k}{\gamma_{k+1}
                \rbr{\frac{\mu}{2} \norm{\yk - \vk}_2^2
                    + \abr{\grad(\yk), \vk - \yk}}}.
        \end{aligned}
    \end{equation}
    Furthermore, the relationship between \( \gamma_{k+1} \) and
    \( \alpha_k \) is the following:
    \begin{equation}
        \label{eq:step-size-relation}
        \gamma_{k+1} = \alpha_{k}^2 / \etak.
    \end{equation}
\end{lemma}
\begin{proof}
    The canonical form for \( \phi_{k+1} \) follows directly from \citet[Lemma
        2.2.3]{nesterov2004lectures}.
    To see the relationship between \( \gamma_{k+1} \) and \( \alpha_k \),
    re-arrange the update for \( \alpha_{k+1} \) in \cref{eq:stochastic-agd}
    to obtain,
    \begin{equation}
        \frac{\alpha_{k}^2}{\etak}
        = (1 - \alpha_{k}) \gamma_k + \alpha_{k} \mu.
    \end{equation}
    By comparison to the update for \( \gamma_k \), we deduce that \( \gamma_{k+1}
    = \alpha_{k}^2 / \etak \).
\end{proof}

\begin{lemma}
    \label{lemma:lambda-convergence}
    Assume \( \alpha_k \in (0, 1) \) and \( \etamin \leq \etak \leq 1 / \mu \)
    almost surely for all \( k \in \bbN \).
    If \( \mu > 0 \) and \( \gamma_0 = \mu \), then
    \begin{equation}
        \label{eq:lambda-sc}
        \lambda_k \leq \prod_{i=0}^{k-1}(1 - \sqrt{\etak \mu}).
    \end{equation}
    Alternately, if \( \gamma_0 \in (\mu, \mu + 3 / \etamin) \), we obtain
    \begin{equation}
        \label{eq:lambda-convex}
        \lambda_k \leq \frac{4}{\etamin (\gamma_0 - \mu)(k+1)^2}.
    \end{equation}
\end{lemma}
\begin{proof}
    \textbf{Case 1}: \( \gamma_0 = \mu > 0 \).
    Then \( \gamma_k = \mu \) for all \( k \) and
    \begin{align*}
        \alpha_{k}^2
         & = (1 - \alpha_{k}) \etak \mu + \alpha_k \etak \mu \\
         & = \etak \mu.
    \end{align*}
    Thus, we deduce that \[ \lambda_k = \prod_{i=0}^{k-1} (1 - \sqrt{\etak \mu}),
    \] and if \( \etak > \etamin \), then \( \alpha_k \geq \sqrt{\etamin \mu} \)
    and \[ \lambda_k \leq (1 - \sqrt{\etamin \mu})^k.
    \]

    \textbf{Case 2}: \( \gamma_0 \in \rbr{\mu, 3 L + \mu} \).
    Using the update rule for \( \gamma_k \) in \cref{lemma:canonical-form},
    we find
    \begin{align*}
        \gamma_{k+1} - \mu
         & = (1 - \alpha_k) \gamma_k + (\alpha_k - 1) \mu
        = (1 - \alpha_k) (\gamma_k - \mu)                  \\
        \intertext{
            Recursing on this equality implies
        }
        \gamma_{k+1}
         & = (\gamma_0 - \mu) \prod_{i=0}^k (1 - \alpha_k)
        = \lambda_{k+1} (\gamma_0 - \mu).
    \end{align*}
    Similarly, using \( \lambda_{k+1} = (1 - \alpha_k) \lambda_k \)
    and \( \alpha_k^2 / \gamma_{k+1} = \etak \)
    yields
    \begin{align*}
        1 - \frac{\lambda_{k+1}}{\lambda_k}
         & = \alpha_k
        = \rbr{\gamma_{k+1} \etak}^{1/2}                                     \\
         & = \rbr{\etak \mu + \etak \lambda_{k+1} (\gamma_0 - \mu)}^{1/2}    \\
        \implies
        \frac{1}{\lambda_{k+1}} - \frac{1}{\lambda_k}
         & = \frac{1}{\lambda_{k+1}^{1/2}}
        \sbr{\frac{\etak \mu}{\lambda_{k+1}} + \etak (\gamma_0 - \mu)}^{1/2} \\
         & \geq \frac{1}{\lambda_{k+1}^{1/2}}
        \sbr{\frac{\etamin \mu}{\lambda_{k+1}} + \etamin (\gamma_0 - \mu)}^{1/2}.
    \end{align*}
    Finally, this implies
    \begin{align*}
        \frac{2}{\lambda_{k+1}^{1/2}}
        \rbr{\frac{1}{\lambda_{k+1}^{1/2}} - \frac{1}{\lambda_k^{1/2}}}
         & \geq \rbr{\frac{1}{\lambda_{k+1}^{1/2}} - \frac{1}{\lambda_k^{1/2}}}
        \rbr{\frac{1}{\lambda_{k+1}^{1/2}} + \frac{1}{\lambda_k^{1/2}}}         \\
         & \geq
        \frac{1}{\lambda_{k+1}^{1/2}}\sbr{\frac{\etamin \mu}{\lambda_{k+1}} + \etamin (\gamma_0 - \mu)}^{1/2}.
    \end{align*}
    Moreover, this bound holds uniformly for all \( k \in \bbN \).
    We have now exactly reached Eq.
    2.2.11 of \citet[Lemma 2.2.4]{nesterov2004lectures}
    with \( L \) replaced by \( \etamin \).
    Applying that Lemma with this modification, we obtain \[ \lambda_k \leq
        \frac{4}{\etamin (\gamma_0 - \mu)(k+1)^2}, \] which completes the proof.

\end{proof}

\sequenceChoice*
\begin{proof}
    Recalling the update for $\alpha_k$, we obtain, \[ \alpha_{k}^2 = (1 -
        \alpha_{k}) \gamma_{k} \etak + \alpha_{k} \etak \mu.
    \]
    Define \( \hat L_k = 1 / \etak \) to obtain the following
    quadratic formula,
    \[
        \hat L_k \alpha_k^2 + (\gamma_k - \mu) \alpha_k - \gamma_k = 0.
    \]
    Using the quadratic equation, we find \[ \alpha_k = \frac{\mu - \gamma_k \pm
            \sqrt{(\mu - \gamma_k)^2 + 4 \hat L_k \gamma_k}}{2 \hat L_k}.
    \]
    As a result, \( \alpha_k > 0 \) if and only if
    \begin{align*}
        (\mu - \gamma_k) + \rbr{(\mu - \gamma_k)^2 + 4 \hat L_k \gamma_k}^{1/2}
                                                                          & > 0.
        \intertext{If \( \mu \geq \gamma_k \), then this holds trivially.
        Otherwise, we require, } (\mu - \gamma_k)^2 + 4 \hat L_k \gamma_k & > (\mu -
        \gamma_k)^2,
    \end{align*}
    which holds if and only if \( \etak, \gamma_k > 0
    \).
    Similarly, \( \alpha_k < 1 \) if and only if
    \begin{align*}
        4 \hat{L}_k^2 + 4\hat L_k (\gamma_k - \mu) + (\mu - \gamma_k)^2
         & > (\mu - \gamma_k)^2 + 4 \hat L_k \gamma_k
        \iff
        \etak < \frac{1}{\mu}.
    \end{align*}
    Since this condition holds by assumption, we have \( \alpha_k \in (0, 1) \)
    for all \( k \).

    Recall \( \lambda_0 = 1 \) and \( \lambda_{k+1} = (1 - \alpha_k) \lambda_k \).
    Since \( \alpha_{k} \in (0, 1) \), \( \lambda_k \geq 0 \) holds by induction.
    It remains to show that \( \lambda_k \) tends to zero.
    Invoking \cref{lemma:lambda-convergence} establishes this result almost
    surely.

    Finally, we must show, \[ \phi_k(w) \leq (1 - \lambda_k) f(w) + \lambda_k
        \phi_0(w).
    \]
    We proceed by induction.
    Since \( \lambda_0 = 1 \), we immediately obtain \[ \phi_0(w) = (1 -
        \lambda_0) f(w) + \lambda_0 \phi_0(w).
    \]
    Now assume the condition holds at \( \phi_k \);
    by the construction of \( \phi_{k+1} \), we have
    \begin{align*}
        \phi_{k+1}(w)
         & = (1 - \alpha_k) \phi_k(w) + \alpha_k \sbr{f(\yk) + \abr{\grad(yk), w - \yk} + \frac{\mu}{2} \norm{w - \yk}} \\
         & \leq (1 - \alpha_k) \phi_k(w) + \alpha_k f(w)                                                                \\
         & \leq (1 - \alpha_k) \sbr{(1 - \lambda_k)f(w) + \lambda_k \phi_0(w)} + \alpha_k f(w)                          \\
         & = \lambda_{k+1} \phi_0(w) - \lambda_{k+1} f(w) + (1 - \alpha_k) f(w) + \alpha_k f(w)                         \\
         & = \lambda_{k+1} \phi_0(w) + (1 - \lambda_{k+1}) f(w).
    \end{align*}
    This completes the proof.
\end{proof}

\localUpperBound*
\begin{proof}
    The choice of \( \phi^*_0 = f(x_0) \) ensures \( \inf \phi_0(\w) = f(\w_0) \)
    deterministically, which is the base case for induction.
    The inductive assumption is \( \E[\inf \phi_k(\w)] \geq \E[ f(\wk) ] \); let
    us use this to show \[ \E\sbr{\inf_{\w} \phi_{k+1}(\w)} \geq \E\sbr{f(\wkk)}.
    \]
    \cref{lemma:canonical-form} implies that the explicit form of the minimizer
    \( \inf \phi_{k+1}(\w) = \phi_{k+1}^* \) is
    \begin{align*}
        \phi_{k+1}^*
                                                                               & = (1- \alpha_k) \phi^*_k + \alpha_k f(\yk)
        - \frac{\alpha_k^2}{2\gamma_{k+1}}\norm{\grad(\yk)}^2                                                                                            \\
                                                                               & \hspace{5em} + \frac{\alpha_k(1-\alpha_k)\gamma_k}{\gamma_{k+1}}
        \rbr{\frac{\mu}{2} \norm{\yk - v_k}^2 + \abr{\grad(\yk), v_k - \yk}}                                                                             \\
        \intertext{
            Taking expectations with respect to \( \z_{0}, \ldots, \z_k \) and
            using linearity of expectation:
        }
        \E[\phi_{k+1}^*]
                                                                               & = \E[(1- \alpha_k) \phi^*_k] + \E\sbr{\alpha_k f(\yk)
        - \frac{\alpha_k^2}{2\gamma_{k+1}}\norm{\grad(\yk)}^2}                                                                                           \\
                                                                               & \hspace{5em} + \E\sbr{\frac{\alpha_k(1-\alpha_k)\gamma_k}{\gamma_{k+1}}
        \rbr{\frac{\mu}{2} \norm{\yk - v_k}^2 + \abr{\grad(\yk), v_k - \yk}}}                                                                            \\
                                                                               & \geq \E\sbr{(1- \alpha_k) f(\wk)} + \E\sbr{\alpha_k f(\yk)
        - \frac{\alpha_k^2}{2\gamma_{k+1}}\norm{\grad(\yk)}^2}                                                                                           \\
                                                                               & \hspace{5em} + \E\sbr{\frac{\alpha_k(1-\alpha_k)\gamma_k}{\gamma_{k+1}}
            \rbr{\frac{\mu}{2} \norm{\yk - v_k}^2 + \abr{\grad(\yk), v_k - \yk}}},
        \intertext{
            where the inequality follows from the inductive assumption.
            Convexity of \( f \) implies \( f(\wk) \geq f(\yk) + \abr{ \grad(\yk), \wk -
                \yk } \).
            Recalling \( \frac{\alpha_k^2}{\gamma_{k+1}} = \etak \) from
        \cref{lemma:canonical-form} allows us to obtain, } \E[\phi_{k+1}^*]    & \geq
        \E[(1- \alpha_k)\rbr{f(\yk) + \abr{\grad(\yk), \wk - \yk}}] + \E\bigg[\alpha_k
        f(\yk) - \frac{\etak}{2}\norm{\grad(\yk)}^2 \bigg]                                                                                               \\  & \hspace{5em} +
           \E\sbr{\frac{\alpha_k(1-\alpha_k)\gamma_k}{\gamma_{k+1}} \rbr{\frac{\mu}{2}
        \norm{\yk - v_k}^2 + \abr{\grad(\yk), v_k - \yk}}}                                                                                               \\  & = \E\sbr{f(\yk)} +
           \E[(1-\alpha_k) \abr{\grad(\yk), \wk - \yk}] -
        \E\bigg[\frac{\etak}{2}\norm{\grad(\yk)}^2 \bigg]                                                                                                \\  & \hspace{5em} +
           \E\sbr{\frac{\alpha_k(1-\alpha_k)\gamma_k}{\gamma_{k+1}} \rbr{\frac{\mu}{2}
        \norm{\yk - v_k}^2 + \abr{\grad(\yk), v_k - \yk}}}                                                                                               \\  & = \E\sbr{f(\yk) -
               \frac{\etak}{2} \norm{\grad(\yk)}^2} + \E\bigg[(1-\alpha_k)
        \bigg(\abr{\grad(\yk), \wk - \yk}                                                                                                                \\  & \hspace{5em} +
               \frac{\alpha_k\gamma_k}{\gamma_{k+1}}\rbr{\frac{\mu}{2} \norm{\yk - v_k}^2 +
        \abr{\grad(\yk), v_k - \yk}}\bigg)\bigg]                                                                                                         \\ \intertext{ The sufficient
        progress condition (\cref{eq:progress}) now implies } \E[\phi_{k+1}^*] & \geq
        \E\sbr{f(\wkk)} + \E\bigg[(1-\alpha_k) \bigg(\abr{\grad(\yk), \wk - \yk}                                                                         \\  &
               \hspace{5em} + \frac{\alpha_k\gamma_k}{\gamma_{k+1}}\rbr{\frac{\mu}{2}
                \norm{\yk - v_k}^2 + \abr{\grad(\yk), v_k - \yk}}\bigg)\bigg].
        \\
        \intertext{
            The remainder of the proof is largely unchanged from the deterministic case.
            The definition of \( \yk \) gives \( \wk - \yk = \frac{\alpha_k
                \gamma_k}{\gamma_{k} + \alpha_k \mu} (\wk - \vk) \), which we use to obtain }
        \E\sbr{\phi_{k+1}^*}                                                   & \geq \E[f(\wkk)] + \E\bigg[(1-\alpha_k)
            \bigg(\frac{\alpha_k \gamma_k}{\gamma_k + \alpha_k \mu} \abr{\grad(\yk), \wk -
        \vk}                                                                                                                                             \\  & \hspace{5em} +
               \frac{\alpha_k\gamma_k}{\gamma_{k+1}}\Big(\frac{\mu}{2} \norm{\yk - v_k}^2 +
        \abr{\grad(\yk), v_k - \yk}\Big)\bigg)\bigg]                                                                                                     \\ \intertext{ Noting that \( v_k
            - \yk = \frac{\gamma_{k+1}}{\gamma_k + \alpha_k \mu}\rbr{v_k - \wk} \) gives }
        \E\sbr{\phi_{k+1}^*}                                                   & \geq \E[f(\wkk)] + \E\bigg[(1-\alpha_k)
            \bigg(\frac{\alpha_k \gamma_k}{\gamma_k + \alpha_k \mu} \abr{\grad(\yk), \wk -
        v_k}                                                                                                                                             \\  & \hspace{5em} + \frac{\alpha_k\gamma_k}{\gamma_{k+1}}
            \Big(\frac{\mu}{2} \norm{\yk - v_k}^2 + \abr{\grad(\yk), v_k -
        \yk}\Big)\bigg)\bigg]                                                                                                                            \\  & = \E[f(\wkk)] + \E\bigg[(1-\alpha_k)
               \bigg(\frac{\alpha_k \gamma_k}{\gamma_k + \alpha_k \mu} \abr{\grad(\yk), \wk -
        v_k}                                                                                                                                             \\  & \hspace{5em} +
               \frac{\alpha_k\gamma_k}{\gamma_{k+1}}\Big(\frac{\mu}{2} \norm{\yk - v_k}^2 +
               \frac{\gamma_{k+1}}{\gamma_k + \alpha_k \mu} \abr{\grad(\yk), v_k -
        \wk}\Big)\bigg)\bigg]                                                                                                                            \\  & = \E[f(\wkk)] + \E\sbr{\frac{\mu
        \alpha_k(1-\alpha_k)\gamma_k}{2 \gamma_{k+1}}\norm{\yk - v_k}^2}                                                                                 \\  & \geq
           \E[f(\wkk)].
    \end{align*}
    since \( \frac{\mu \alpha_k(1-\alpha_k)\gamma_k}{2 \gamma_{k+1}} \geq 0 \).
    We conclude that \( \E[\inf \phi_{k}(\w)] \geq \E[f(\wk)] \) holds for all \(
    k \in \bbN \) by induction.
\end{proof}

\linearConvergence*
\begin{proof}
    Under the conditions of the theorem, \cref{prop:local-upper-bound} holds
    and \cref{eq:estimating-convergence} implies
    \begin{align*}
        \E\sbr{f(\wk)} - f(\wopt)
         & \leq \E\sbr{\lambda_k(\phi_0(\wopt) - f(\wopt))}
        \intertext{
            Since \( \gamma_0 = \mu > 0 \), \cref{lemma:lambda-convergence} implies
            \( \lambda_k = \prod_{i=0}^{k-1} (1 - \sqrt{\eta_i \mu}) \) and
            we obtain
        }
        \E\sbr{f(\wk)} - f(\wopt)
         & \leq \E\sbr{\prod_{i=0}^{k-1} (1 - \sqrt{\eta_i \mu})} (\phi_0(\wopt) - f(\wopt)) \\
         & = \E\sbr{\prod_{i=0}^{k-1} (1 - \sqrt{\eta_i \mu})}
        \sbr{f(w_0) - f(\wopt) + \frac{\mu}{2} \norm{w_0 - \wopt}_2^2},
    \end{align*}
    which completes the proof.
\end{proof}

\sublinearConvergence*
\begin{proof}
    Under the conditions of the theorem, \cref{prop:local-upper-bound} holds
    and \cref{eq:estimating-convergence} implies
    \begin{align*}
        \E\sbr{f(\wk)} - f(\wopt)
         & \leq \E\sbr{\lambda_k(\phi_0(\wopt) - f(\wopt))}
        \intertext{
            Since \( \gamma_0 \in (\mu, \mu + 3/\etamin) \), \cref{lemma:lambda-convergence} implies
            \( \lambda_k \leq 4 / \etamin (\gamma_0 - \mu) (k+1)^2 \) and
            we obtain
        }
        \E\sbr{f(\wk)} - f(\wopt)
         & \leq \frac{4}{\etamin (\gamma_0 - \mu) (k+1)^2} (\phi_0(\wopt) - f(\wopt)) \\
         & = \frac{4}{\etamin (\gamma_0 - \mu) (k+1)^2}
        \sbr{f(w_0) - f(\wopt) + \frac{\gamma_0}{2} \norm{w_0 - \wopt}_2^2},
    \end{align*}
    which completes the proof.
\end{proof}

\subsection{Specializations: Proofs}

\begin{lemma}
    \label{lemma:matrix-strong-growth}
    Let \( D \in \R^{d \times d} \) be independent of \( \zk \).
    Suppose \( f \) is convex and \( L \)-smooth, the stochastic gradients \(
    \grad(\wk, \zk) \) are \( \Lmax^D \) individually smooth with respect to the
    matrix norm \( \norm{\cdot}_D \), and interpolation holds.
    If \( f \) is also \( \mu_D \)-strongly convex
    with respect to \( \norm{\cdot}_D \), then
    \begin{equation}
        \Ek \sbr{\norm{\grad(\w,\zk)}_{D^{-1}}^2}
        \leq \frac{\Lmax^D}{\mu_{D}} \norm{\grad(w)}_{D^{-1}}^2.
    \end{equation}
    That is, strong growth holds in \( \norm{\cdot}_D \) with constant \( \rho_D
    \leq \frac{\Lmax^D}{\mu_D} \).
\end{lemma}
\begin{proof}
    Starting from \( \Lmax^D \) individual-smoothness,
    \begin{align*}
        f(u, \zk)
                                                       & \leq f(\w, \zk) + \abr{\grad(\w, \zk), u - \w}
        + \frac{\Lmax^D}{2}\norm{u - \w}_D^2,                                                           \\
        \intertext{
            and choosing \( u = \w - \frac{1}{\Lmax^D}
        D^{-1}\grad(\w, \zk) \), we obtain } f(u, \zk) & \leq f(\w, \zk) -
        \frac{1}{\Lmax^D} \abr{\grad(\w, \zk), D^{-1} \grad(\w,\zk)} + \frac{1}{2
        \Lmax^D} \norm{D^{-1}\grad(\w,\zk)}_D^2                                                         \\  & = f(\w, \zk) - \frac{1}{2
               \Lmax^D}\norm{\grad(\w,\zk)}_{D^{-1}}^2.
    \end{align*}
    Noting that \( f(u, \zk) \geq f(\wopt, \zk) \) by convexity
    of \( f \) and interpolation and taking expectations with respect
    to \( \zk \) gives the following:
    \begin{align*}
        f(\wopt, \zk)
         & \leq f(\w, \zk) - \frac{1}{2 \Lmax^D} \norm{\grad(\w,\zk)}_{D^{-1}}^2                  \\
        \implies \Ek\sbr{f(\wopt, \zk)}
         & \leq \Ek\sbr{f(\w, \zk)} - \frac{1}{2 \Lmax^D}\Ek\sbr{\norm{\grad(\w,\zk)}_{D^{-1}}^2} \\
        \implies f(\wopt)
         & \leq f(\w) - \frac{1}{2 \Lmax^D}\Ek\sbr{\norm{\grad(\w,\zk)}_{D^{-1}}^2}.
    \end{align*}
    Re-arranging this final equation and using \( \mu_D \)-strong convexity
    gives the desired result,
    \begin{align*}
        \Ek \sbr{\norm{\grad(\w,\zk)}_{D^{-1}}^2}
         & \leq 2 \Lmax^D \rbr{f(\w) - f(\wopt)}                  \\
         & \leq \frac{\Lmax^D}{\mu_D} \norm{\grad(w)}_{D^{-1}}^2.
    \end{align*}
\end{proof}

\begin{restatable}{lemma}{matrixDescentLemma}\label{lemma:matrix-descent-lemma}
    Let \( D \in \R^{d \times d} \).
    Assume the \( f \) is both \( L_D \)-smooth and satisfies \( \rho_D \) strong
    growth in the matrix norm \( \norm{\cdot}_D \).
    If \( 0 \prec D \preceq I \) and \( \etak \leq \frac{1}{\rho_D L_D} \), then
    preconditioned SGD satisfies, \[ \E_{\zk}\sbr{f(\wkk)} \leq f(\yk) -
        \frac{\etak}{2}\norm{\grad\rbr{\yk}}^2.
    \]
\end{restatable}
\begin{proof}
    Starting from smoothness in \( \norm{\cdot}_D \),
    \begin{align*}
        f(\wkk)
         & \leq f(\yk) + \abr{\grad(\yk), \wkk - \yk}
        + \frac{L_D}{2}\norm{\wkk - \yk}_D^2                                       \\
         & \leq f(\yk) - \etak \abr{\grad(\yk), D^{-1} \grad\rbr{\yk, \zk}}
        + \frac{\etak^2 L_D}{2}\norm{\grad\rbr{\yk, \zk}}_{D^{-1}}^2.
        \intertext{Taking expectations with respect to \( \zk \),}
        \implies
        \E_{\zk}\sbr{f(\wkk)}
         & \leq f(\yk) - \etak \abr{\grad(\yk), D^{-1}_k \grad\rbr{\yk}}
        + \frac{\etak^2 L_D}{2}\E_{\zk}\sbr{\norm{\grad\rbr{\yk, \zk}}_{D^{-1}}^2} \\
         & \leq f(\yk) - \etak \abr{\grad(\yk), D^{-1}_k \grad\rbr{\yk}}
        + \frac{\etak^2 \rho_{D}
        L_D}{2}\norm{\grad\rbr{\yk}}_{D^{-1}}^2                                    \\  & = f(\yk) - \etak \rbr{1 -
        \frac{\etak \rho_D L_D}{2}} \norm{\grad\rbr{\yk}}_{D^{-1}}^2               \\  & \leq f(\yk)
        - \frac{\etak}{2}\norm{\grad\rbr{\yk}}_{D^{-1}}^2                          \\  & \leq f(\yk) -
           \frac{\etak}{2}\norm{\grad\rbr{\yk}}_2^2.
    \end{align*}
\end{proof}

%% file: appendices/comparison_proofs.tex
%!TEX root=../main.tex

\uniformBasis*
\begin{proof}
    It is easy to see that \( \Lmax = 1 \).
    Taking expectations, we find that \[ \fls(w) = \frac{1}{2n}\norm{w}_2^2, \]
    which implies \( L = \mu = 1 / n \).
    It is also straightforward to compute \( \rho \):
    \begin{align*}
        \E\sbr{\norm{\gradls(w, \zk)}_2^2}
         & = \frac{1}{n} \sum_{i=1}^n \norm{e_i (e_i^\top w)}^2_2
        = \frac{1}{n} \sum_{i=1}^n w_i^2
        = \frac{1}{n} \norm{w}_2^2
        = n \norm{\gradls(w)}_2^2,
    \end{align*}
    which implies \( \rho = n \).
    Note that this is tight with the bound \( \rho \leq \Lmax / \mu \).

    Now we compute the values of \( \kappa \) and \( \tilde \kappa \).
    Since the Hessian satisfies \( H = I/n \), it is straightforward to see that
    \begin{align*}
        \E_x\sbr{\norm{x}_2^2 (x x^\top)}
         & = \frac{1}{n} \sum_{i=1}^n \norm{e_i}_2^2 e_i e_i^\top = H    \\
        \E_x\sbr{\norm{x}^2_{H^{-1}} (x x^\top)}
         & = \frac{1}{n} \sum_{i=1}^n \norm{e_i}_{H^{-1}}^2 e_i e_i^\top
        = \sum_{i=1}^n \norm{e_i}_{2}^2 e_i e_i^\top = n H,
    \end{align*}
    which implies that \( \kappa = \tilde \kappa = n \).
    Substituting these values into the complexity bounds for stochastic AGD and
    MaSS completes the example.
\end{proof}

\nonUniformBasis*
\begin{proof}
    Again, it is easy to see that \( \Lmax = 1 \).
    Taking expectations, we find that \[ \fls(w) = \frac{n - 1}{2n} w_1^2 +
        \frac{1}{2n} w_2^2, \] which implies \( L = \frac{n-1}{n} \) and \( \mu =
    \frac{1}{n} \).
    The strong growth constant is given by
    \begin{align*}
        \E\sbr{\norm{\grad(w, \zk)}_2^2}
         & = \frac{n-1}{n} w_1^2 + \frac{1}{n} w_2^2                    \\
         & = n \rbr{\frac{n-1}{n^2} w_1^2 + \frac{1}{n^2} w_2^2}        \\
         & \leq n \rbr{\frac{(n-1)^2}{n^2} w_1^2 + \frac{1}{n^2} w_2^2} \\
         & = n \norm{\grad(w)}_2^2,
    \end{align*}
    which implies \( \rho \leq n \).
    It's easy to see that this is tight by taking \( w_1 = 0 \) and \( w_2 \neq 0
    \).

    Now we compute the values of \( \kappa \) and \( \tilde \kappa \).
    The Hessian is given by
    \[
        H =
        \begin{bmatrix}
            \frac{n - 1}{n} & 0           \\
            0               & \frac{1}{n}
        \end{bmatrix}
        ,
    \]
    and thus
    \begin{align*}
        \E_x\sbr{\norm{x}_2^2 (x x^\top)}
         & = \frac{n-1}{n} e_1 e_1^\top + \frac{1}{n} e_2 e_2^\top = H \\
        \E_x\sbr{\norm{x}^2_{H^{-1}} (x x^\top)}
         & = e_1 e_1^\top + e_2 e_2^\top = I \preceq n H,
    \end{align*}
    which implies \( \tilde \kappa = \kappa = n \).
    Substituting these values into the complexity bounds for stochastic AGD and
    MaSS completes the example.
\end{proof}

%% file: appendices/bug.tex
%!TEX root=../main.tex

As stated in the preamble, the optimization schemes given in
\cref{eq:stochastic-agd} and \cref{eq:stochastic-agd-simple} are not
equivalent as \cref{sec:convergence} claims.
While the formal equivalence of the deterministic versions of these two
algorithms is proved by \citet[Eq. 2.2.20]{nesterov2004lectures}, their argument
fails in our setting because \cref{eq:stochastic-agd} combines a deterministic
estimating sequence with stochastic gradient updates.
In particular, the deterministic estimating sequence implies deterministic
gradient updates for \( \vk \), which cannot reconciled with the stochastic
updates for \( \xk \) when using the standard argument.
See \cref{sec:equivalence} for sketch of how the proof fails.

This bug does not falsify all results in this preprint.
Since the estimating sequence form of AGD in \cref{eq:stochastic-agd} is the target
of our theoretical analysis in \cref{sec:convergence},
all of the claimed convergence rates continue to hold for this scheme.
They fail only for the momentum version of stochastic AGD in
\cref{eq:stochastic-agd-simple}.
However, the estimating sequence form of AGD is not truly stochastic because,
as noted, \cref{eq:stochastic-agd} requires deterministic gradient updates for
\( \vk \).
Thus, we are only able to prove much weaker results than the original claims in
the preprint.

There are several potential ways to resolve this problem.
One approach is to change the estimating sequence to use stochastic gradient
updates.
That is, change \cref{eq:estimating-sequence} to,
\begin{equation}\label{eq:stochastic-estimating-sequence}
    \begin{aligned}
        \lambda_{k+1}
         & = (1 - \alpha_k) \lambda_k            \\
        \phi_{k+1}(w)
         & = (1 - \alpha_k) \phi_k(w) + \alpha_k
        (f(\yk) + \abr{\grad(\yk, \zk), w - \yk}
        + \frac{\mu}{2}\norm{w - \yk}_2^2).
    \end{aligned}
\end{equation}
Since \( \vkk \) is the minimizer of \( \phi_{k+1} \), this results in a
stochastic gradient update for \( \vkk \) in \cref{eq:stochastic-agd}.
Analyzing this fully stochastic version is straightforward if strong growth is
used to control the \( \norm{\nabla f(\yk, \zk)}_2^2 \) term which now appears
in \( \phi_{k+1}^* \).
Disappointingly, this approach leads to the \( \rho \) dependence
(as opposed to \( \sqrt{\rho} \)) previously established by
\citet{vaswani2019fast}.

Another approach is to retain a deterministic estimating sequence
while still modifying \cref{eq:stochastic-agd} to use a stochastic update
for \( \vk \).
In that case, we obtain two different \( \vk \) sequences: the true minimizers
of \( \phi_k \) and a sequence of unbiased stochastic estimates \( \tilde v_k
\) maintained by the algorithm.
While promising, this approach leads to the expectation
\( \E\sbr{\abr{\grad(\yk), \vk - \tilde{v}_k}} \) in the analysis of
\cref{prop:local-upper-bound}.
The issue here is that \( \tilde{v}_k \) is correlated with \( \grad(\yk) \)
as \( \yk \) is computed from \( \tilde{v}_k \), meaning this expectation
does not resolve to zero.
It is not clear whether or not a more careful analysis of this term can yield
the desired \( \sqrt{\rho} \) dependence.

Another option is that obtaining a \( \sqrt{\rho} \) dependence for stochastic
AGD is not possible.
That is, there is a lower bound showing that the \( O(\rho\sqrt{L/\mu}
\log(1/\epsilon)) \) complexity proved \citet{vaswani2019fast} is in fact tight.
We are investigating this possibility.

\subsection{Failed Equivalence Argument}\label{sec:equivalence}

Now we show how the standard equivalence argument fails.
Our goal is to eliminate the \( \vk \) sequence by writing \( \ykk \)
as a function of \( \yk \) and \( \wkk \).
Recalling that,
\[
    \yk
    = \frac{1}{\gamma_{k} + \alpha_{k}\mu}\sbr{\alpha_{k} \gamma_{k} \vk + \gamma_{k+1}\wk},
\]
we can re-arrange to obtain the following expression for \( \vk \),
\[
    \vk
    = \frac{1}{\alpha_k \gamma_k} ((\gamma_{k} + \alpha_k \mu) \yk - \gamma_{k+1} \wk).
\]
Then, starting from the definition of \( \vkk \) in \cref{eq:stochastic-agd}
and substituting in this value for \( \vk \), we obtain,
\begin{align*}
    \vkk
     & = \frac{1}{\gamma_{k+1}}\sbr{\frac{(1-\alpha_k)}{\alpha_k}
    ((\gamma_k + \alpha_k \mu)\yk - \gamma_{k+1}\wk) + \alpha_k\mu \yk - \alpha_k \grad(\yk)} \\
     & = \frac{(1-\alpha_k)\gamma_k}{\alpha_k \gamma_{k+1}}\yk + \frac{\mu}{\gamma_{k+1}} \yk
    - \frac{(1-\alpha_k)}{\alpha_k}\wk - \frac{\alpha_k}{\gamma_{k+1}} \grad(\yk)             \\
     & = \frac{1}{\alpha_k}\yk
    - \frac{(1-\alpha_k)}{\alpha_k}\wk - \frac{\etak}{\alpha_k} \grad(\yk)                    \\
     & = \wk + \frac{1}{\alpha_k}\rbr{\yk - \etak \grad(\yk) - \wk}                           \\
     & = \wk + \frac{1}{\alpha_k}\rbr{\wkk - \wk}
    + \textcolor{red}{\frac{\etak}{\alpha_k}\rbr{\grad(\yk, \zk) - \grad(\yk)}}.
\end{align*}
The proof now proceeds by substituting this expression for \( \vkk \)
into the equation for \( \ykk \).
Although the error term \( \frac{\etak}{\alpha_k}\rbr{\grad(\yk, \zk) - \grad(\yk)} \)
is expectation zero, it cannot be removed.
This breaks the proof that the schemes in \cref{eq:stochastic-agd} and
\cref{eq:stochastic-agd-simple} are equivalent.